\documentclass{amsproc}
\usepackage{amssymb}
\usepackage{amsfonts}

\theoremstyle{plain}

\newtheorem{corollary}{Corollary}

\newtheorem{definition}{Definition}

\newtheorem{lemma}{Lemma}

\newtheorem{proposition}{Proposition}
\newtheorem{remark}{Remark}

\newtheorem{theorem}{Theorem}
\numberwithin{equation}{section}

\begin{document}
\title[Smoothness of orbital measures]{The smoothness of orbital measures on
noncompact symmetric spaces}
\author{Sanjiv Kumar Gupta}
\address{Dept. of Mathematics \\
Sultan Qaboos University\\
P.O.Box 36 Al Khodh 123\\
Sultanate of Oman}
\email{gupta@squ.edu.om}
\author{Kathryn E. Hare}
\address{Dept. of Pure Mathematics\\
University of Waterloo\\
Waterloo, Ont.,~Canada\\
N2L 3G1}
\email{kehare@uwaterloo.ca}
\thanks{This research is supported in part by NSERC 2016-03719 and by Sultan
Qaboos University. The authors thank Acadia University for their hospitality
when this research was done. }
\thanks{This paper is in final form and no version of it will be submitted
for publication elsewhere.}

\begin{abstract}
Let $G/K$ be an irreducible symmetric space where $G$ is a non-compact,
connected Lie group and $K$ is a compact, connected subgroup. We use decay
properties of the spherical functions to show that the convolution product
of any $r=r(G/K)$ continuous orbital measures has its density function in $%
L^{2}(G)$ and hence is an absolutely continuous measure with respect to Haar
measure. The number $r$ is approximately the rank of $G/K$. For the special
case of the orbital measures, $\nu _{a_{i}}$, supported on the double cosets 
$Ka_{i}K$ where $a_{i}$ belongs to the dense set of regular elements, we
prove the sharp result that $\nu _{a_{1}}\ast \nu _{a_{2}}\in L^{2},$ except
for the symmetric space of Cartan type $AI$ when the convolution of three
orbital measures is needed (even though $\nu _{a_{1}}\ast \nu _{a_{2}}$ is
absolutely continuous).
\end{abstract}

\maketitle

\section{Introduction}

Let $G$ be a real, connected, noncompact, semisimple Lie group with finite
center, and $K$ a maximal compact subgroup of $G$. The quotient space, $G/K,$
is a symmetric space of noncompact type, which we also assume to be
irreducible. For $a\in G\diagdown N_{G}(K),$ we let $\nu _{a}$ denote the $K$%
-bi-invariant, orbital, singular measure supported on the compact double
coset $KaK$ in $G$. The smoothness properties of convolution products of
these orbital measures has been of interest for many years and is related to
questions about products of double cosets and spherical functions. Ragozin,
in \cite{Ra}, proved that for $r\geq $ $\dim G/K$, the convolution product
measure, $\nu _{a_{1}}\ast \cdots \ast \nu _{a_{r}},$ is absolutely
continuous with respect to any Haar measure on $G$, equivalently, its
density function is a compactly supported function in\ $L^{1}(G)$. This was
improved in a series of papers, culminating with \cite{GSFunc} and \cite%
{GHJMAA17}, where $r$ was reduced to either $rankG/K$ or $rankG/K+1$
depending on the Lie type. See \cite{GS1} for a good history of this problem.

For the special case of regular elements, $a_{j},$ it was shown in \cite%
{AReg} that the density function of $\nu _{a_{1}}\ast \cdot \cdot \cdot \ast
\nu _{a_{r}}$ belongs to the smaller space of compactly supported functions
in $L^{2}(G)$ for $r\geq $ $\dim G/K+1$. The decay properties of spherical
functions and the Plancherel theorem were used to prove this. In this paper,
we develop a more refined analysis of the decay properties of spherical
functions, using the classification of these symmetric spaces in terms of
their restricted root systems, to significantly improve this result. This
analysis allows us to both extend the $L^{2}$ result to convolutions of 
\textit{all }orbital measures $\nu _{a}$ for $a\notin N_{G}(K)$, as well as
to reduce the number of convolution products to approximately $rankG/K;$ the
precise values are given in Section 4 and depend only on the Lie and Cartan
type of the symmetric space. In the special case of convolution products of
orbital measures at regular elements, we prove that $r=2$ suffices, except
for one symmetric space (Cartan type $AI$ of rank one) where $r=3$ is both
necessary and sufficient. This latter fact shows that, unlike the situation
for the analogous problem in compact Lie groups and algebras, it is not true
that $\nu _{a}^{k}$ belongs to $L^{2}$ if and only if $\nu _{a}^{k}$ is
absolutely continuous (where the exponent means convolution powers). The
decay properties are also applied to study the differentiability of orbital
measures.

\section{Notation and Basic Facts}

\subsection{Lie algebra set up}

Let $G$ be a real, connected, non-compact, semisimple Lie group with finite
centre and let $K$ be a maximal compact subgroup of $G$ fixed by the Cartan
involution $\theta $. We assume that $G/K$ is irreducible. The quotient
space, $G/K$, is a symmetric space of non-compact type III in Helgason's
terminology, \cite{He}. Let $\mathfrak{g=t\oplus p}$ be the corresponding
Cartan decomposition of the Lie algebra $\mathfrak{g}$ of $G,$ where $%
\mathfrak{t}$ is the Lie algebra of $K$ and $\mathfrak{p}$ is the orthogonal
complement of $\mathfrak{t}$ with respect to the Killing form of $\mathfrak{g%
}$. We fix a maximal abelian (as a subalgebra of $\mathfrak{g}$) subspace $%
\mathfrak{a}$ of $\mathfrak{p}$ and let $\mathfrak{a}^{\ast }$ denote its
dual. The rank of $G/K$ is the dimension of $\mathfrak{a}$. If we put $%
A=\exp \mathfrak{a}$ where $\exp $: $\mathfrak{g}\rightarrow G$ is the
exponential function, then $G=KAK$.

The set of restricted roots, $\Phi ,$ is defined by 
\begin{equation*}
\Phi =\{\alpha \in \mathfrak{a}^{\ast }:\mathfrak{g}_{\alpha }\neq 0\}
\end{equation*}%
where $\mathfrak{g}_{\alpha }$ are the root spaces. The multiplicity of the
restricted root $\alpha $ will be denoted 
\begin{equation*}
m_{\alpha }=\dim \mathfrak{g}_{\alpha }.
\end{equation*}
The subset of positive restricted roots is denoted $\Phi ^{+}$. The set $%
\Phi $ is a root system, although not necessarily reduced as it is possible
for both $\alpha $ and $2\alpha $ to be in $\Phi $.

Take a basis $\mathcal{B}$ for $\mathfrak{a}^{\ast }$ consisting of positive
simple roots and let $\mathfrak{a}^{+}$ be the elements $H\in \mathfrak{a}$
with $\alpha (H)>0$ for all $\alpha \in \mathcal{B}$. Similarly, let $%
\mathcal{D}\subseteq \mathfrak{a}$ be the dual basis to $\mathcal{B}$ and
let 
\begin{equation*}
\mathfrak{a}^{\ast +}=\{\lambda \in \mathfrak{a}^{\ast }:\lambda (H)>0\text{ 
}\forall H\in \mathcal{D\}}.
\end{equation*}%
We have $\mathfrak{a}^{\ast }=\bigcup_{w\in W}w(\overline{\mathfrak{a}%
^{\ast +}})$ for $W$ equal to the Weyl group, with a similar statement
holding for $\mathfrak{a}$.

Consequently, $G=K\overline{\exp \mathfrak{a}^{+}}K$. Indeed, given any $%
g\in G,$ there is a pair $k_{1},k_{2}\in K$ and a unique $X_{g}\in \overline{%
\mathfrak{a}^{+}}$ such that $g=k_{1}(\exp X_{g})k_{2}$. We can thus view $%
\lambda \in \mathfrak{a}^{\ast }$ as also acting on $g\in A$ by setting $%
\lambda (g)=\lambda (X_{g})$.

The symmetric spaces can be classified by their Cartan class and the Lie
type of their restricted root system, these being one of types $A_{n},$ $%
B_{n},$ $C_{n},$ $BC_{n},$ and $D_{n}$ (the classical types) or $G_{2},F_{4},
$ $E_{6},$ $E_{7},E_{8}$ (the exceptional types), the subscript in all cases
being the rank of the symmetric space. We remark that for types $B_{n},C_{n}$
we may assume $n\geq 2$ as the symmetric spaces of Lie types $B_{1}$ and $%
C_{1}$ are isomorphic to type $A_{1}$. Similarly, with $D_{n}$ we may assume 
$n\geq 4$. For more details, please see the appendix.

For further background on this material and proofs of the facts stated above
we refer the reader to \cite{Hediff}-\cite{Hu}.

\subsection{Orbital measures}

Next, we introduce the orbital measures of interest in this paper. We let $%
dm $ denote normalized Haar measure on $K$.

\begin{definition}
Let $a\in A$. By an \textbf{orbital measure on }$G,$ we mean the measure
denoted $\nu _{a},$ defined by the rule%
\begin{equation*}
\int_{G}f(g)d\nu _{a}(g)=\int_{K}\int_{K}f(k_{1}ak_{2})dm(k_{1})dm(k_{2})
\end{equation*}%
for all continuous, compactly supported functions $f$ on $G$.
\end{definition}

The orbital measure $\nu _{a}$ is the $K$-bi-invariant, probability measure
supported on the compact, double coset $KaK\subseteq G$. Orbital measures
are always singular with respect to Haar measure on $G$ and they are
continuous measures (i.e., have no atoms) when $a\notin N_{G}(K)$, the
normalizer of $K$ in $G$.

It is a classical problem to study the smoothness of convolution products of
continuous orbital measures. Some of the earliest work was done by Ragozin
in \cite{Ra} who showed that $\nu _{a_{1}}\ast \cdot \cdot \cdot \ast \nu
_{a_{r}}$ is absolutely continuous if and only if the product of double
cosets, $Ka_{1}Ka_{2}\cdot \cdot \cdot Ka_{r}K,$ has non-empty interior in $%
G $. He, then, used geometric arguments to prove that the latter statement
was true whenever $r\geq \dim G/K$. Using algebraic methods, this was
subsequently improved to $r\geq rankG/K+1$ by Graczyk and Sawyer in \cite%
{GSFunc}, who also showed that this was sharp in the case of non-compact
symmetric spaces with restricted root systems of type $A_{n}$. Inspired by
Graczyk and Sawyer's work in \cite{GSLie} and \cite{GSAus}, in \cite%
{GHJMAA17} the authors proved that $r\geq rankG/K$ is the sharp $L^{1}$
result for all the classical non-compact symmetric spaces except those of
type $A_{n},$ and characterized precisely which convolution products are
absolutely continuous for the classical types.

\subsection{ $L^{1}-L^{2}$ Dichotomy}

Similar smoothness questions have been explored in a number of related
settings, including $K$-bi-invariant measures supported on double cosets in
compact symmetric spaces $G/K$, invariant measures supported on conjugacy
classes of compact Lie groups or $Ad$-invariant measures supported on
adjoint orbits of compact Lie algebras. In the case of compact Lie groups
and algebras, the authors in \cite{GHDich} and \cite{GHMathZ} used a
combination of harmonic analysis and geometric arguments to show that
convolution powers of such measures belong to $L^{1}$ if and only they
belong to $L^{2},$ and determined the sharp exponent for each such measure.
In contrast, in \cite{AGP} it was shown that this dichotomy fails to hold in
the compact symmetric space $SU(2)/SO(2)$.

The harmonic analysis approach to the $L^{2}$ problem for compact Lie groups
involved studying the rate of decay of the characters of the group and
applying the Plancherel theorem. For symmetric spaces, the analogous
approach is to study, instead, the decay of the spherical transform. We
recall the definition of the spherical function and spherical transform.

\begin{definition}
The \textbf{spherical transform} of a compactly supported measure $\nu $ on
the non-compact Lie group $G$ is defined by 
\begin{equation*}
\widehat{\nu }(\lambda )=\int_{G}\phi _{\lambda }(g^{-1})d\nu (g)
\end{equation*}%
where $\phi _{\lambda }$ is the \textbf{spherical function} corresponding to 
$\lambda \in \mathfrak{a}^{\ast }$ given by the expression%
\begin{equation*}
\phi _{\lambda }(g)=\int_{K}\exp ((i\lambda -\rho )\mathcal{H}(gk))dm(k).
\end{equation*}%
Here $\rho $ is half the sum of the positive roots and $\mathcal{H}$ is the
Iwasawa projection, i.e., $\mathcal{H(}gk)$ is the unique element in $%
\mathfrak{a}$ such that $gk\in K\exp \mathcal{H}(gk)N$ where $N$ is a Lie
subgroup of \ $G$ with Lie algebra $\mathfrak{n=}\sum_{\alpha \in \Phi ^{+}}%
\mathfrak{g}_{\alpha }$.
\end{definition}

This formula for the spherical function can be found in \cite[IV, Thm. 4.3]%
{He} where it is also seen that $\phi _{\lambda }=\phi _{w(\lambda )}$ for
all $w\in W$ and $\lambda \in \mathfrak{a}^{\ast }$.

From the definition of orbital measures it is easy to see that $\widehat{\nu
_{a}}(\lambda )=\phi _{\lambda }(a^{-1}),$ while in \cite{AReg} it is shown
that 
\begin{equation*}
(\nu _{a_{1}}\ast \cdot \cdot \cdot \ast \nu _{a_{r}})\widehat{}(\lambda
)=\prod_{i=1}^{r}\phi _{\lambda }(a_{i}^{-1}).
\end{equation*}

%The spherical inversion formula \cite[IV Thm. 7.5]{He} states that for
%`nice' $K$-bi-invariant functions $f$ on $G$ we have%
%\begin{equation*}
%f(g)=\frac{1}{\left\vert W\right\vert }\int_{\mathfrak{a}^{\ast }}\widehat{f}%
%(\lambda )\phi _{\lambda }(g)\left\vert c\left( \lambda \right) \right\vert
%^{-2}d\lambda
%\end{equation*}%

A version of Plancherel's theorem holds in this setting. For the remainder
of the paper, $c=c(\lambda )$ is the Harish-Chandra $c$ function and $%
d\lambda $ denotes Lebesgue measure on $\mathfrak{a}^{\ast }$.

\begin{theorem}
(Plancherel) (see \cite[IV Thm. 9.1]{He}) The $K$-bi-invariant measure $\mu $
belongs to $L^{2}(G)$ if and only if%
\begin{equation*}
\left\Vert \mu \right\Vert _{L^{2}(G)}^{2}=\int_{\mathfrak{a}^{\ast
}}\left\vert \widehat{\mu }(\lambda )\right\vert ^{2}\left\vert c(\lambda
)\right\vert ^{-2}d\lambda <\infty .
\end{equation*}
\end{theorem}

\begin{corollary}
The $k$-fold convolution product of the orbital measure $\nu _{a}$ belongs
to $L^{2}(G)$ if and only if $\left\vert \phi _{\lambda }(a)\right\vert
^{k}\left\vert c(\lambda )\right\vert ^{-1}\in L^{2}(\mathfrak{a}^{\ast })$.
\end{corollary}

It is known that the spherical functions have good decay properties. To
explain, it is helpful to introduce further terminology and notation.

\begin{definition}
(i) Given $a\in A$ (or $a\in \mathfrak{a}$), by the \textbf{set of
annihilating roots} of $a$ we mean the set%
\begin{equation*}
\Phi (a)=\{\alpha \in \Phi :\alpha (a)=0\}.
\end{equation*}%
Put $\Phi ^{+}(a)=\Phi (a)\cap \Phi ^{+}$. By $(\Phi ^{+}(a))^{c}$ we mean
the complement of $\Phi (a)$ in $\Phi ^{+}$, that is, $(\Phi ^{+}(a))^{c}$ $%
= $ $\{\alpha \in \Phi ^{+}:\alpha (a)\neq 0\}$.

(ii) If $\Phi (a)$ is empty, the element $a$ is called \textbf{regular}. If $%
a$ is regular, we call $\nu _{a}$ a \textbf{regular orbital measure}.
\end{definition}

We will let 
\begin{equation*}
A_{0}=\{g\in A:g\notin N_{G}(K)\}.
\end{equation*}%
The set $N_{G}(K)$ can be characterized as the set of elements $g\in G$ such
that $\alpha (g)=0$ for all roots $\alpha $, hence the set of annihilating
roots of an element in $A_{0}$ is a proper root subsystem. The set of
regular elements is dense in $A$ and in the special case of a rank one
symmetric space all the elements of $A_{0}$ are regular.

Here is the decay result that we will use.

\begin{proposition}
(\cite[Thm. 11.1]{DKV}, see also \cite[Prop. 4.1]{AReg}) For each $a\in
A_{0},$ there is a constant $C_{a}$ such that for all $\lambda \in \mathfrak{%
a}^{\ast },$ 
\begin{equation}
\left\vert \phi _{\lambda }(a)\right\vert \leq C_{a}\sum_{w\in
W}\prod_{\alpha \in (\Phi ^{+}(w(a)))^{c}}\left( 1+\left\vert \langle
\lambda ,\alpha \rangle \right\vert \right) ^{-m_{\alpha }/2}\text{.}
\label{Ca}
\end{equation}
\end{proposition}

It is well known (see \cite[IV.7.2 ]{He}) that there is a constant $C$ such
that 
\begin{equation*}
|c(\lambda )|^{-1}\leq C\prod_{\alpha \in \Phi ^{+}}\left( 1+\left\vert
\langle \lambda ,\alpha \rangle \right\vert \right) ^{m_{\alpha }/2},
\end{equation*}%
thus%
\begin{equation}
(\left\vert \phi _{\lambda }(a)\right\vert ^{k}|c(\lambda )|^{-1})^{2}\leq
C_{a}\max_{w\in W}\prod_{\alpha \in (\Phi ^{+}(w(a)))^{c}}\left\vert
1+\left\vert \langle \lambda ,\alpha \rangle \right\vert \right\vert
^{-m_{\alpha }k}\prod_{\alpha \in \Phi ^{+}}\left\vert 1+\left\vert \langle
\lambda ,\alpha \rangle \right\vert \right\vert ^{m_{\alpha }}  \label{BdPhi}
\end{equation}%
for a new constant $C_{a}$. Combined with Plancherel's theorem, this implies 
$\nu _{a}^{k}$ belongs to $L^{2}(G)$ provided%
\begin{equation}
\int_{\overline{\mathfrak{a}^{\ast +}}}\max_{w\in W}\prod_{\alpha \in (\Phi
^{+}(w(a)))^{c}}\left\vert 1+\langle \lambda ,\alpha \rangle \right\vert
^{-m_{\alpha }k}\prod_{\alpha \in \Phi ^{+}}\left\vert 1+\langle \lambda
,\alpha \rangle \right\vert ^{m_{\alpha }}d\lambda <\infty \text{.}
\label{Int}
\end{equation}

\section{$L^{2}$ results for convolutions of orbital measures at regular
elements}

In \cite{AReg}, bounds were found for the right hand side of (\ref{BdPhi})
that were sufficient to show that any convolution product of more than $\dim
G/K$ regular orbital measures was in $L^{2}(G)$. We will begin by improving
this result, in fact, obtaining sharp $L^{2}$ results for convolution
products of regular orbital measures.

\begin{theorem}
\label{regular}Suppose $a\in A_{0}$ is a regular element. The convolution
products, $\nu _{a}^{k}$, belong to $L^{2}(G)$ if and only if $k\geq 2,$
except if the symmetric space $G/K$ has restricted root system of type $%
A_{1} $ and is of Cartan class $AI$, in which case $k\geq 3$ is both
necessary and sufficient.
\end{theorem}

\begin{remark}
We remark that $k\geq 2$ is necessary since $\nu _{a}$ is always a singular
measure.
\end{remark}

We will first obtain bounds for $\left\vert \phi _{\lambda }(a)\right\vert
^{k}|c(\lambda )|^{-1}$ for the symmetric spaces of classical Lie types. Let 
$\eta _{0}$ denote the multiplicity of the standard roots $e_{i}\pm e_{j},$ $%
\eta _{1}$ the multiplicity of the short roots $e_{i},$ and $\eta _{2}$ the
multiplicity of the long roots $2e_{i}$ (should there be roots of these
forms). The reader can find the values of $\eta _{j}$ for each type in the
appendix.

\begin{lemma}
\label{L:reg}Suppose the restricted root system of $G/K$ is one of the Lie
types $A_{n},$ $B_{n},$ $C_{n}$, $BC_{n}$ or $D_{n}$ and that $a\in A_{0}$
is a regular element. There is a positive constant $C,$ depending only on $%
G/K$ and $a$, such that%
\begin{equation*}
\left( \left\vert \phi _{\lambda }(a)\right\vert ^{k}|c(\lambda
)|^{-1}\right) ^{2}\leq C\min \left( 1,\left\Vert \lambda \right\Vert
^{(1-k)\varrho }\right) \text{ for all }\lambda \in \mathfrak{a}^{\ast }%
\text{ and }k\geq 1\text{,}
\end{equation*}%
where%
\begin{equation*}
\varrho =\varrho (G/K)=\left\{ 
\begin{array}{cc}
\eta _{0}n & \text{for Lie type }A_{n} \\ 
\begin{array}{c}
\eta _{0}(2n-3)+\eta _{1}+\eta _{2} \\ 
\eta _{1}+\eta _{2}%
\end{array}
& 
\begin{array}{c}
\text{for Lie types }B_{n},C_{n},BC_{n},n\geq 2 \\ 
\text{ for Lie type }BC_{1}%
\end{array}
\\ 
\eta _{0}2(n-1)\text{ } & \text{for Lie type }D_{n}%
\end{array}%
\right. .
\end{equation*}
\end{lemma}

\begin{proof}
Throughout the proof, the constant $C$ may vary from one occurrence to
another. We will assume $G/K$ has rank $n$ and there is no loss of
generality in assuming $\lambda \in $ $\overline{\mathfrak{a}^{\ast +}}$.

As $a$ is regular, $\Phi (w(a))$ is empty for all $w\in W$ and thus 
\begin{equation}
\left( \left\vert \phi _{\lambda }(a)\right\vert ^{k}|c(\lambda
)|^{-1}\right) ^{2}\leq C\prod_{\alpha \in \Phi ^{+}}\left\vert 1+\langle
\lambda ,\alpha \rangle \right\vert ^{m_{\alpha }(1-k)}.  \label{spdecay}
\end{equation}%
Of course, if $\left\Vert \lambda \right\Vert \leq 1$, then $\prod_{\alpha
\in \Phi ^{+}}\left\vert 1+\langle \lambda ,\alpha \rangle \right\vert \leq
C,$ so our interest is in $\left\Vert \lambda \right\Vert \geq 1$.

We will let 
\begin{equation}
T_{\lambda }=\{\alpha \in \Phi ^{+}:\langle \alpha ,\lambda \rangle \geq
c_{G}\left\Vert \lambda \right\Vert \}  \label{T}
\end{equation}%
where the choice of constant $c_{G}>0$ will depend on the Lie type and will
be made clear later in the proof. We will let $S_{0}=\{e_{i}\pm e_{j}:1\leq
i<j\leq n\}$, $S_{1}=\{e_{i}:1\leq i\leq n\}$ and $S_{2}=\{2e_{i}:1\leq
i\leq n\}$ (should they exist). For example, in type $A_{n}$, $S_{0}=\Phi
^{+}$ and $S_{1},S_{2}$ do not exist. Notice $m_{\alpha }=\eta _{j}$ if $%
\alpha \in S_{j}$. Put 
\begin{equation}
U_{\lambda ,j}=T_{\lambda }\cap S_{j}  \label{U}
\end{equation}%
and write $\left\vert U_{\lambda ,j}\right\vert $ for the cardinality of
this set.

With this notation, we have%
\begin{equation}
\left( \left\vert \phi _{\lambda }(a)\right\vert ^{k}|c(\lambda
)|^{-1}\right) ^{2}\leq C\min \left( 1,\left\Vert \lambda \right\Vert
^{(1-k)\sum_{j}\eta _{j}\left\vert U_{\lambda ,j}\right\vert }\right) .
\label{RegBd1}
\end{equation}

We will find lower bounds on $\left\vert U_{\lambda ,j}\right\vert $ by
analyzing on a type-by-type basis.

\medskip

Type $A_{n}$: We can write $\lambda =\sum_{i=1}^{n}a_{i}\lambda _{i}$ where $%
\lambda _{i}$ are the fundamental dominant weights (the dual basis to the
basis of simple roots) and $a_{i}\geq 0$. Since all norms are equivalent on
a finite dimensional normed space, we can take $\left\Vert \lambda
\right\Vert =\max_{i}a_{i}=a_{m}$ (say). It will suffice to determine which
positive roots $\alpha =\sum_{i=1}^{n}b_{i}\alpha _{i}$ have $b_{m}>0$ (and
hence $b_{m}\geq 1)$ for then $\langle \alpha ,\lambda \rangle
=\sum_{i}a_{i}b_{i}\geq a_{m}b_{m}\geq \left\Vert \lambda \right\Vert $ and $%
U_{\lambda ,0}$ will contain that set of roots. (Here we will take $c_{G}=1$%
.) These will be the roots $\alpha =e_{i}-e_{j}$ where $1\leq i\leq m$ and $%
m<j\leq n+1$, thus the minimum value of $\left\vert U_{\lambda
,0}\right\vert $ is $n$.

\smallskip

Type $B_{n},$ $C_{n},$ $BC_{n}$: We leave the very easy case of $BC_{1}$ to
the reader and assume $n\geq 2$. Here we can write $\lambda
=\sum_{i=1}^{n}a_{i}e_{i}$ where $a_{i}\geq 0$ are non-increasing, and $e_{i}
$ are the standard basis vectors for $\mathbb{R}^{n}$. Taking the Euclidean
norm, we have $a_{1}\leq \left\Vert \lambda \right\Vert \leq na_{1}$. We
have $\langle \alpha ,\lambda \rangle \geq a_{1}$ if $\alpha =e_{1}+e_{j}$
for $j=2,...,n$ or $\alpha =(2)e_{1}$. In particular, for any choice of $%
c_{G}\leq 1$ we have $\left\vert U_{\lambda ,j}\right\vert \geq 1$ for $j=1$
in type $B_{n}$, for $j=2$ for type $C_{n}$ and for both $j=1,2$ for type $%
BC_{n}$. If $a_{2}\leq a_{1}/2,$ then we also have $\langle \alpha ,\lambda
\rangle \geq a_{1}/2$ if $\alpha =e_{1}-e_{j}$ for $j=2,...,n$. In this
case, $\left\vert U_{\lambda ,0}\right\vert \geq 2(n-1)$. Otherwise, $%
a_{2}>a_{1}/2,$ and then $\langle \alpha ,\lambda \rangle \geq a_{1}/2$ if $%
a=e_{2}+e_{j}$, $j=3,...,n$ or $\alpha =(2)e_{2}$. In this case, taking $%
c_{G}=1/2$, we have $\left\vert U_{\lambda ,0}\right\vert \geq 2n-3$ and $%
\left\vert U_{\lambda ,1}\right\vert \geq 2$ for type $B_{n},$ with similar
statements for $C_{n}$ and $BC_{n}$. Of course, the minimum of $2\eta
_{0}(n-1)+\eta _{1}+\eta _{2}$ and $\eta _{0}(2n-3)+2(\eta _{1}+\eta _{2})$
is at least $\eta _{0}(2n-3)+(\eta _{1}+\eta _{2})$.

\smallskip

Type $D_{n}$: As with type $A_{n}$, we write $\lambda
=\sum_{i=1}^{n}a_{i}\lambda _{i}$ where $\lambda _{i}$ are the fundamental
dominant weights and $a_{i}\geq 0$. We again take $c_{G}=1$. It suffices to
determine which $\alpha =\sum b_{i}\alpha _{i}$ have $b_{m}>0,$ where $%
a_{m}=\max_{i}a_{i}$. If $m\neq n-1,n,$ these will be the roots $\alpha
=e_{i}+e_{j}$ for $i\leq m$ and $j>i$ and for $\alpha =e_{i}-e_{j}$ for $%
i\leq m<j$. There are at least $2(n-1)$ of these roots. If $m=n,$ all the
roots $e_{i}+e_{j}$ have the desired property, while if $m=n-1,$ the
positive roots $e_{i}-e_{n}$, $i<n$ and $e_{i}+e_{j}$, $i<j<n$ all work.
Thus for all $\lambda ,$ $\left\vert U_{\lambda ,0}\right\vert \geq \min
\left( 2(n-1),\binom{n}{2}\right) =2(n-1)$ as we may assume $n\geq 4$ for
this type.
\end{proof}

\begin{proof}
\lbrack of Theorem \ref{regular}] We begin by proving the sufficiency of the
choice of $k$. As in the lemma, the constant $C>0,$ depending on $G/K$ and $%
a,$ which appears throughout may change from one occurrence to another. We
again assume $G/K$ has rank $n$.

When $G/K$ has a restricted root space of classical Lie type, the previous
lemma shows that 
\begin{equation}
\left\Vert \nu _{a}^{k}\right\Vert _{2}^{2}\leq C\int_{\overline{\mathfrak{a}%
^{\ast +}}}\min (1,\left\Vert \lambda \right\Vert ^{(1-k)\varrho })d\lambda
\leq C\int_{1}^{\infty }t^{(1-k)\varrho }t^{n-1}dt  \label{Pl}
\end{equation}%
and this will be finite if $(1-k)\varrho +n-1<-1$. It is a routine exercise,
using the values of $\varrho $ given in the Lemma, to see that if $k\geq 2,$
then this is true for all these classical types, except if $G/K$ is of Lie
type $A_{n}$ and Cartan class $AI$. In this latter case, $\eta _{0}=1$ and
we have that the integral above is finite provided $k\geq 3$.

\medskip

However, the argument can be improved for the Lie type $A_{n},$ Cartan class 
$AI,$ when $n\geq 2$. Let%
\begin{eqnarray*}
\Lambda _{0} &=&\{\lambda =\sum_{i=1}^{n}a_{i}\lambda _{i}\in \overline{%
\mathfrak{a}^{\ast +}}:a_{j}=\max a_{i}\text{ for some }j\neq 1,n\}, \\
\Lambda _{1} &=&\{\lambda =\sum_{i=1}^{n}a_{i}\lambda _{i}\in \overline{%
\mathfrak{a}^{\ast +}}:a_{1}=a_{n}=\max a_{i}\}
\end{eqnarray*}%
and let $\Lambda _{2}$ be the rest of $\overline{\mathfrak{a}^{\ast +}}$.
Note that 
\begin{equation*}
\left\Vert \nu _{a}^{2}\right\Vert _{2}^{2}\leq C\sum_{j=0}^{2}\int_{\Lambda
_{j}}\left( \left\vert \phi _{\lambda }(a)\right\vert ^{2}|c(\lambda
)|^{-1}\right) ^{2}d\lambda .
\end{equation*}

Let $U_{\lambda ,0}$ be as in the lemma. Note that $\left\vert U_{\lambda
,0}\right\vert \geq n+1$ if $\lambda \in \Lambda _{0}\cup \Lambda _{1}$,
from whence one can see that $\int_{\Lambda _{j}}\left( \left\vert \phi
_{\lambda }(a)\right\vert ^{2}|c(\lambda )|^{-1}\right) ^{2}d\lambda <\infty 
$ for $j=0,1$.

If, instead $\lambda \in \Lambda _{2}$ (so either $a_{1}$ or $a_{n}$ is the
unique maximal coordinate), then $\left\vert U_{\lambda ,0}\right\vert =n$.
However, there will also be at least $n-1$ positive roots $\alpha \notin
U_{\lambda ,0}$ such that $\langle \alpha ,\lambda \rangle \geq a_{J}$,
where $a_{J}$ is the second largest coefficient. Using this fact, we obtain
the bound 
\begin{eqnarray*}
\int_{\Lambda _{2}}\left( \left\vert \phi _{\lambda }(a)\right\vert
^{2}|c(\lambda )|^{-1}\right) ^{2}d\lambda &\leq &C\int_{0}^{\infty
}(1+t_{1})^{-n}\left(
\int_{0}^{t_{1}}(1+t_{2})^{-(n-1)}t_{2}^{n-2}dt_{2}\right) dt_{1} \\
&\leq &C\left( 1+\int_{1}^{\infty
}t_{1}^{-n}\int_{1}^{t_{1}}t_{2}^{-1}dt_{1}\right) \\
&=&C\left( 1+\int_{1}^{\infty }t_{1}^{-n}\log t_{1}dt_{1}\right)
\end{eqnarray*}%
and this is finite since we are assuming $n\geq 2$.

Thus even when the symmetric space is of Cartan class $AI,$ we have $\nu
_{a}^{2}\in L^{2}$ provided the rank of $G/K$ is at least $n=2$. That
completes the proof of sufficiency of the choice of $k$ for the classical
Lie types.

\medskip

For the symmetric spaces with restricted root spaces of exceptional Lie
types, we argue in a similar fashion. We define $T_{\lambda }$ as in (\ref{T}%
) and decompose the set of positive restricted roots into maximal disjoint
sets $S_{j},$ consisting of the positive roots of a given multiplicity.
Again, put $U_{\lambda ,j}=T_{\lambda }\cap S_{j}$ and observe that again (%
\ref{RegBd1}) holds.

If the restricted root space is Lie type $G_{2},E_{6},E_{7}$ or $E_{8}$,
then all the roots have the same multiplicity, so we take $S_{0}=\Phi ^{+}$.
It is shown in \cite{HS} (see, for example, Tables 2,3,4) that the minimum
cardinality of $U_{\lambda ,0}$ is at least $5,16,27$ and $57$ respectively.

If the restricted root space is Lie type $F_{4}$ and all the roots have the
same multiplicity, again $S_{0}=\Phi ^{+}$ and the minimum cardinality of $%
U_{\lambda ,0}$ is shown in \cite{HS} to be $15$. Otherwise, there are two
distinct multiplicities and we define $S_{0},S_{1}$ accordingly. As can be
seen from \cite{HS}, $\left\vert U_{\lambda ,0}\right\vert \geq 9$ and $%
\left\vert U_{\lambda ,1}\right\vert \geq 6$. Using (\ref{Pl}) again, it is
easy to check that $k\geq 2$ suffices.

\medskip

We turn now to proving the necessity of the choice of $k$. Since $\nu _{a}$
is a singular measure with respect to Haar measure, $k\geq 2$ is certainly
necessary (in all cases). Thus we need only consider the symmetric space $%
G/K $ of Lie type $A_{1}$ and Cartan class $AI$ and show that $\nu _{a}^{2}$
does not belong to $L^{2}$.

For this symmetric space, the spherical functions can be expressed in terms
of the hypergeometric functions $_{2}F_{1}$ as follows. Denote by $\alpha $
the (single) positive root and choose $H_{0}\in \mathfrak{a}$ such that $%
\alpha (H_{0})=1$. For any $t\neq 0$, it is known (\cite[11.5.15]{Wo}) that 
\begin{equation*}
\phi _{\lambda }(\exp tH_{0})=_{2}F_{1}\left( \frac{1+i\lambda }{4},\frac{%
1-i\lambda }{4},1,-\sinh ^{2}t\right) .
\end{equation*}%
Next, we use the relationship between the hypergeometric functions and the
Jacobi and Bessel functions (c.f., \cite[Sec. 6.4]{FH}):%
\begin{equation*}
J_{u}^{(0,b)}(t)=_{2}F_{1}\left( \frac{b+1+iu}{2},\frac{b+1-iu}{2},1,-\sinh
^{2}t\right) ,
\end{equation*}%
while 
\begin{equation*}
J_{u}^{(0,b)}(t)=cJ_{0}(ut)+O(u^{-3/2}),
\end{equation*}%
where $J_{0}(\cdot )$ is the Bessel function and $c$ is a non-zero constant
depending on $t$. It is well known (\cite[9.2.1]{AS}) that for $z$ $>0,$ 
\begin{equation*}
J_{0}(z)=\frac{C}{\sqrt{z}}\left( \cos (z-\pi /4)+O(z^{-1}\right) )
\end{equation*}%
for some $C\neq 0$. Thus for all $\lambda >0,$%
\begin{equation}
\phi _{\lambda }(\exp tH_{0})=\frac{C}{\sqrt{\lambda }}\cos (\lambda t/2-\pi
/4)+O(\left\vert \lambda \right\vert ^{-3/2})  \label{PhiAsy}
\end{equation}%
where the non-zero constant $C$ depends only on $t$.

For any integer $j,$ let $I_{j}$ denote the interval 
\begin{equation*}
I_{j}=\frac{2}{t}\left[ (2j+1)\frac{\pi }{2}+\frac{\pi }{8},(2j+1)\frac{\pi 
}{2}+\frac{3\pi }{8}\right]
\end{equation*}%
and let $\bigcup_{j\in \mathbb{Z}}I_{j}=I^{\ast }$. If $\lambda \notin
I^{\ast }$, then $\left\vert C\cos (\lambda t/2-\pi /4)\right\vert \geq
\left\vert C\cos 3\pi /8\right\vert =\varepsilon _{0}>0$. The asymptotic
formula for $\phi _{\lambda },$ (\ref{PhiAsy}), shows that we may choose $%
\lambda _{1}$ sufficiently large so that for all $\lambda \geq \lambda _{1}$
with $\lambda \notin I^{\ast },$ we have 
\begin{equation*}
\left\vert \phi _{\lambda }(\exp tH_{0})\right\vert \geq \frac{\varepsilon
_{0}}{2\sqrt{\lambda }}.
\end{equation*}%
It is shown in the proof of Prop. 7.2 in \cite{He}, that for the Harish
Chandra $c$ function, $\lim_{\lambda \rightarrow \infty }c(\lambda
)^{-1}\lambda ^{-1/2}=2\sqrt{\pi }$. Thus $c(\lambda )^{-1}\geq \sqrt{\pi
\lambda }$ for all $\lambda \geq \lambda _{2},$ say. Let $\lambda _{0}=\max
(\lambda _{1},\lambda _{2})$. Putting these bounds together shows that 
\begin{equation*}
\int \left\vert \phi _{\lambda }^{2}(\exp tH_{0})c(\lambda )^{-1}\right\vert
^{2}d\lambda \geq \int_{\lambda \notin I^{\ast },\lambda \geq \lambda
_{0}}\left( \frac{\varepsilon _{0}}{2\sqrt{\lambda }}\right) ^{4}\left( 
\sqrt{\pi \lambda }\right) ^{2}d\lambda =C\int_{\lambda \notin I^{\ast
},\lambda \geq \lambda _{0}}\frac{d\lambda }{\lambda }.
\end{equation*}%
Choose $k_{0}$ such that $(2k-1)\pi /2+3\pi /8\geq \lambda _{0}$ for $k\geq
k_{0}$ and set 
\begin{equation*}
L_{k}=\frac{2}{t}[(2k-1)\frac{\pi }{2}+\frac{3\pi }{8},(2k+1)\frac{\pi }{2}+%
\frac{\pi }{8}].
\end{equation*}%
We deduce that 
\begin{equation*}
\int \left\vert \phi _{\lambda }^{2}(\exp tH_{0})c(\lambda )^{-1}\right\vert
^{2}d\lambda \geq C\sum_{k=k_{0}}^{\infty }\int_{L_{k}}\frac{d\lambda }{%
\lambda }\geq C\sum_{k=k_{0}}^{\infty }\frac{\text{length}(L_{k})}{k}=\infty
.
\end{equation*}%
Consequently, $\phi _{\lambda }^{2}(\exp tH_{0})c(\lambda )^{-1}\notin L^{2}$
and that proves $\nu _{a}^{2}\notin L^{2}$ for any $a=\exp tH_{0},t\neq 0,$
and hence for any regular $a$.
\end{proof}

\begin{remark}
It is known that for any non-compact, rank $1$ symmetric space, $\nu
_{a}\ast \nu _{a}$ belongs to $L^{1}$ for all $a\in A_{0}$ (\cite{GSFunc}).
Thus the $L^{1}-L^{2}$ dichotomy fails for the symmetric space of Lie type $%
A_{1}$ and Cartan class $AI$. Interestingly, the $L^{1}-$ $L^{2}$ dichotomy
holds for all the regular orbital measures in all the other symmetric spaces
since we obviously have $\nu _{a}^{k}\in L^{1}$ only if $k\geq 2$.
\end{remark}

\begin{corollary}
\label{Creg}Let $a_{1},a_{2},a_{3}$ be regular elements in $A$. If $G/K$ is
Lie type $A_{1}$ and Cartan class $AI$, then $\nu _{a_{1}}\ast \nu
_{a_{2}}\ast \nu _{a_{3}}\in L^{2}$. Otherwise, $\nu _{a_{1}}\ast \nu
_{a_{2}}\in L^{2}$.
\end{corollary}

\begin{proof}
We will prove the first statement as the second is even easier. Let $\mu =$ $%
\nu _{a_{1}}\ast \nu _{a_{2}}\ast \nu _{a_{3}}$. By the Plancherel formula,%
\begin{equation*}
\left\Vert \mu \right\Vert _{2}^{2}=\int_{\mathfrak{a}^{\ast }}\left\vert 
\widehat{\mu }(\lambda )\right\vert ^{2}\left\vert c(\lambda )\right\vert
^{-2}d\lambda =\int_{\mathfrak{a}^{\ast }}\left\vert \prod_{i=1}^{3}\phi
_{\lambda }(a_{i}^{-1})\right\vert ^{2}\left\vert c(\lambda )\right\vert
^{-2}d\lambda .
\end{equation*}%
Applying the generalized Holder's inequality gives%
\begin{equation*}
\left\Vert \mu \right\Vert _{2}^{2}\leq \prod_{i=1}^{3}\left( \int_{%
\mathfrak{a}^{\ast }}\left\vert \phi _{\lambda }(a_{i}^{-1})\right\vert
^{6}\left\vert c(\lambda )\right\vert ^{-2}d\lambda \right)
^{1/3}=\prod_{i=1}^{3}\left\Vert \nu _{a_{i}}^{3}\right\Vert _{2}^{2/3},
\end{equation*}%
and the latter product is finite according to the Theorem.
\end{proof}

\bigskip

\section{\protect\bigskip Smoothness of convolutions of arbitrary orbital
measures}

\subsection{ $L^{2}$ results}

The goal of this section is to show that for all $a\in A_{0}$ (not just
regular $a)$ there is an index $k$ such that $\nu _{a}^{k}\in L^{2}(G)$. As
in the proof of Theorem \ref{regular}, we will continue to use the notation $%
\eta _{0}$ to denote the multiplicity of the roots $e_{i}\pm e_{j},$ $\eta
_{1}$ for the multiplicity of the short roots $e_{i},$ and $\eta _{2}$ for
the multiplicity of the long roots $2e_{i}$ when the symmetric space is of
classical Lie type $A_{n}$, $B_{n}$, $C_{n}$, $BC_{n}$ or $D_{n}$. We recall
that the values of $\eta _{j}$ depend on the Lie type and Cartan class and
can be found in the Appendix.

\begin{theorem}
\label{TMain}Let $G/K$ be a non-compact symmetric space of type $A_{n}$, $%
B_{n}$, $C_{n}$, $D_{n}$ or $BC_{n}$. If $v_{a_{1}},...,v_{a_{k}}$ are any
orbital measures on $G$ with $a_{i}\in A_{0},$ then $v_{a_{1}}\ast \cdot
\cdot \cdot \ast v_{a_{k}}\in L^{2}(G)$ provided $k>k_{G}$ where 
\begin{equation*}
k_{G}=\left\{ 
\begin{array}{cc}
n+n/\eta _{0} & \text{ for type }A_{n} \\ 
n-1+n/(2\eta _{0}) & \text{for type }D_{n} \\ 
\begin{array}{c}
2(n-1)+(n+\eta _{1}+\eta _{2})/\eta _{0} \\ 
\max \left( 4,2+(\eta _{1}+\eta _{2})/(2\eta _{0})\right)%
\end{array}
& 
\begin{array}{c}
\text{for type }B_{n},C_{n},BC_{n},n\geq 3 \\ 
\text{for }B_{2},C_{2},BC_{2}%
\end{array}%
\end{array}%
\right. .
\end{equation*}
\end{theorem}

\begin{remark}
We remark that the symmetric spaces of Lie type $A_{n},$ $(B)C_{n}$ or $%
D_{n} $ have rank $n$ and dimension $O(n(n+\eta _{1}+\eta _{2}))$. Note that
for type $\left( B\right) C_{n}$ we can assume $n\geq 2$ as the regular
orbital measure case has already be done.
\end{remark}

The key to the proof of this theorem is finding bounds for the products%
\begin{equation}
P_{G/K}^{w}(\lambda ,k,a)=\prod_{\alpha \in \Phi ^{+}(w(a))^{c}}\left\vert
1+\langle \lambda ,\alpha \rangle \right\vert ^{-m_{\alpha }k}\prod_{\alpha
\in \Phi ^{+}}\left\vert 1+\langle \lambda ,\alpha \rangle \right\vert
^{m_{\alpha }},  \label{Pnka}
\end{equation}%
and 
\begin{equation*}
P_{G/K}(\lambda ,k,a)=\max_{w\in W}P_{G/K}^{w}(\lambda ,k,a)
\end{equation*}%
for $\lambda \in \overline{\mathfrak{a}^{\ast +}}$ since we have already
seen in (\ref{BdPhi}) that%
\begin{equation*}
\left( \left\vert \phi _{\lambda }(a)\right\vert ^{k}|c(\lambda
)|^{-1}\right) ^{2}\leq C_{a}P_{G/K}(\lambda ,k,a).
\end{equation*}%
This will be mainly accomplished in two lemmas. We will again write $C$ for
a positive constant (depending only on $G/K$ and $a$) that may change
throughout the proof. We begin with the symmetric spaces of Lie type $A_{n}$
or $D_{n}$. These are easier as all roots have the same multiplicity.

\begin{lemma}
\label{LAD}Suppose $G/K$ is Lie type $A_{n-1}$ or $D_{n}$ and $a\in A_{0}$.
There is a constant $C$ such that 
\begin{equation*}
P_{G/K}(\lambda ,k,a)\leq C\min (1,\left\Vert \lambda \right\Vert ^{-\eta
_{0}p_{k})})
\end{equation*}%
for all integers $k\geq n-1$ and $\lambda \in \overline{\mathfrak{a}^{\ast +}%
},$ where%
\begin{equation*}
p_{k}=p_{k}(G/K)=\left\{ 
\begin{array}{cc}
k-n+1 & \text{ for }G/K\text{ type }A_{n-1} \\ 
2(k-n+1) & \text{for }G/K\text{ type }D_{n}%
\end{array}%
\right. .
\end{equation*}
\end{lemma}

\begin{proof}
There is a constant $C$ such that $P_{G/K}^{w}(\lambda ,k,a)\leq C$ if $%
\left\Vert \lambda \right\Vert \leq 1,$ thus our interest is with $%
\left\Vert \lambda \right\Vert \geq 1$.

In \cite{HWY} the analogous problem was studied for the invariant measures
supported on conjugacy classes in the classical simple compact Lie groups.
Specifically, in (3.1) of \cite{HWY}, it was shown that if $\mathcal{G}$ is
a compact Lie group, $X^{+}$ is the set of positive roots for the Lie
algebra associated with $\mathcal{G}$ and $Y^{+}$ is the set of positive
roots of some maximal root subsystem, then for all representations $\lambda $%
, 
\begin{equation}
\prod_{\alpha \in (Y^{+})^{c}}\left\vert 1+\langle \lambda ,\alpha \rangle
\right\vert ^{-1}\leq C\prod_{\alpha \in X^{+}}\left\vert 1+\langle \lambda
,\alpha \rangle \right\vert ^{-s}  \label{Group}
\end{equation}%
where $s=1/(n-1)$ if $\mathcal{G}$ is Lie type $A_{n-1}$ or $D_{n}$. Athough
this was formally shown only for all representations of $\mathcal{G}$, the
same reasoning gives the same bound for all $\lambda \in \overline{\mathfrak{%
a}^{\ast +}}$ with $\left\Vert \lambda \right\Vert \geq 1$.

Consider the compact Lie group $\mathcal{G}$ with the same root system $\Phi 
$ as the restricted root system of $G/K$ (although, with all roots having
multiplicity two, rather than $\eta _{0}$). For any $a\in A_{0}$ and $w\in
W, $ the set of positive annihilating roots of $w(a)$ is contained in the
set of positive roots of a maximal root subsystem of $\Phi ,$ say $\Psi ^{+}$%
. Appealing to (\ref{Group}) we deduce that if $\left\Vert \lambda
\right\Vert \geq 1,$ then%
\begin{eqnarray*}
P_{G/K}^{w}(\lambda ,k,a) &\leq &\left( \prod_{\alpha \in (\Psi
^{+})^{c}}\left\vert 1+\langle \lambda ,\alpha \rangle \right\vert
^{-k}\prod_{\alpha \in \Phi ^{+}}\left\vert 1+\langle \lambda ,\alpha
\rangle \right\vert \right) ^{\eta _{0}} \\
&\leq &C\prod_{\alpha \in \Phi ^{+}}\left\vert 1+\langle \lambda ,\alpha
\rangle \right\vert ^{(1-ks)\eta _{0}}
\end{eqnarray*}%
(for the appropriate choice of $s$). Hence, if we let $q$ be the minimal
number of positive roots $\alpha $ (not counting multiplicity) such that $%
\langle \lambda ,\alpha \rangle \geq c_{G}\left\Vert \lambda \right\Vert $
for such $\lambda ,$ with $c_{G}>0$ as in (\ref{T}), then $P_{G/K}(\lambda
,k,a)\leq C\left\Vert \lambda \right\Vert ^{(1-ks)\eta _{0}q}$. In the
notation of (\ref{U}), $q=$ $\min_{\lambda }\left\vert U_{\lambda
,0}\right\vert .$ Thus $q(A_{n-1})=n-1$ and $q(D_{n})=$ $2(n-1)${\LARGE .}
Inputting the values for $s$ and $q$ gives the desired result.
\end{proof}

\begin{lemma}
\label{LBC}Suppose $G/K$ is Lie type $B_{n}$, $C_{n}$ or $BC_{n},$ $\lambda
\in \overline{\mathfrak{a}^{\ast +}}$ and $a\in A_{0}$.

(i) If $n\geq 3,$ there is a constant $C_{n}$ such that if integer $k\geq
\kappa _{n}:=2(n-1)+(\eta _{1}+\eta _{2})/\eta _{0}$, then 
\begin{equation}
P_{G/K}(\lambda ,k,a)\leq C_{n}\min (1,\left\Vert \lambda \right\Vert ^{\eta
_{0}(2(n-1)-k)+\eta _{1}+\eta _{2}})\text{ }.  \label{Claim1}
\end{equation}%
(ii) Suppose $n=2,$ $m=\min (\eta _{0},\eta _{1}+\eta _{2})$ and $\ M=\max
(\eta _{0},\eta _{1}+\eta _{2})$. Then if integer $k\geq \kappa _{2}=1+M/2m,$
\begin{equation}
P_{G/K}(\lambda ,k,a)\leq C_{2}\min (1,\max (\left\Vert \lambda \right\Vert
^{2m(1-k)+M})\text{ .}  \label{Claim2}
\end{equation}
\end{lemma}

\begin{proof}
As noted previously, we obviously have $P_{G/K}(\lambda ,k,a)$ uniformly
bounded when $\left\Vert \lambda \right\Vert \leq 1$. Moreover, when $n\geq
3 $ and integer $k\geq \kappa _{n},$ then $\eta _{0}(2(n-1)-k)+\eta
_{1}+\eta _{2}\leq 0$ and when $k\geq \kappa _{2},$ $2m(1-k)+M\leq 0$. Thus
the task is to check that $P_{G/K}(\lambda ,k,a)\leq C_{n}\left\Vert \lambda
\right\Vert ^{\eta _{0}(2(n-1)-k)+\eta _{1}+\eta _{2}}$ when $n\geq 3$ and
the corresponding statement of (ii) when $n=2$.

Our proof of (i) will proceed by induction on $n$. We will leave the
arguments for the base case until the end when it will be done in
conjunction with the proof of (ii).

We will give the proof for type $BC_{n},$ but the modifications for the
other types are essentially notational. For the induction argument, it will
is natural to write $P_{n}(\lambda ,k,a)$ rather than $P_{G/K}(\lambda ,k,a)$
when the rank of $G/K$ is $n$.

Let $a\in A_{0}$. Since $\Phi ^{+}(w(a))$ is a proper root subsystem, in
bounding $P_{n}(\lambda ,k,a)$ we may as well assume $\Phi ^{+}(w(a))=\Psi
^{+},$ where $\Psi $ is one of the finitely many maximal root subsystems,
and that $w=id$.

The maximal root subsystems of a symmetric space of Lie type $BC_{n}$ are:
(a) Lie type $BC_{n-1}$, (b) Lie type $A_{n-1}$ and (c) Lie types $%
BC_{n-j}\times A_{j-1}$ with $n-j\geq 1,$ $j\geq 2$.

Any spherical representation in $BC_{n}$ can be written as $\lambda
=\sum_{i=1}^{n}\lambda _{i}e_{i}$ where $\lambda _{i}$ are non-increasing,
non-negative integers. Thus $\lambda _{1}\leq $ $\left\Vert \lambda
\right\Vert \leq n\lambda _{1}$ and, consequently,%
\begin{equation}
\prod_{\alpha \in \Phi ^{+}}\left\vert 1+\langle \lambda ,\alpha \rangle
\right\vert ^{m_{\alpha }}\leq C\lambda _{1}^{2\binom{n}{2}\eta _{0}+n(\eta
_{1}+\eta _{2})}.  \label{All}
\end{equation}%
\medskip

We now consider the three cases of maximal annihilating root subsystems
separately.

\medskip

Case (a) $\Psi $ is of type $BC_{n-1}$: That means there is some index $%
n_{0}\in \{1,...,n\}$ such that 
\begin{equation*}
\Psi ^{+}=\{e_{i}\pm e_{j},e_{k},2e_{k}:1\leq i<j\leq n,\text{ }i,j,k\neq
n_{0}\},
\end{equation*}%
and hence 
\begin{equation*}
(\Psi ^{+})^{c}=\{e_{n_{0}}\pm e_{j},e_{n_{0}},2e_{n_{0}}:j\neq n_{0}\}
\end{equation*}%
(where $e_{n_{0}}-e_{j}$ should be replaced by $e_{j}-e_{n_{0}}$ if $j<n_{0}$%
).

If $n_{0}=1,$ then as $1+\langle \lambda ,e_{1}+e_{j}\rangle \geq \lambda
_{1}$ for all $j=2,...,n$ and $1+\langle \lambda ,(2)e_{1}\rangle \geq
\lambda _{1},$ we see that 
\begin{equation*}
\prod_{\alpha \in (\Psi ^{+})^{c}}\left\vert 1+\langle \lambda ,\alpha
\rangle \right\vert ^{m_{\alpha }}\geq \lambda _{1}^{(n-1)\eta _{0}+\eta
_{1}+\eta _{2}}.
\end{equation*}%
Thus for such $a,$%
\begin{equation}
P_{n}(\lambda ,k,a)\leq \lambda _{1}^{(n-1)\eta _{0}(n-k)+(\eta _{1}+\eta
_{2})(n-k)}  \label{Case1a}
\end{equation}%
and that's dominated by the right hand side of (\ref{Claim1}) when $k\geq
\kappa _{n}$.

\smallskip

So assume $n_{0}\neq 1$.\ Here we will use an induction argument assuming
the statement holds for $n-1$. (Actually, all we will need to inductively
assume is that $P_{n-1}(\lambda ,k,a)$ is uniformly bounded for $k\geq
\kappa _{n}$ and the claims of the lemma certainly ensure this.)

We consider the root subsystem 
\begin{equation*}
\Phi ^{\prime }=\{e_{i}\pm e_{j},e_{k},2e_{k}:2\leq i\neq j\leq n,2\leq
k\leq n\}\subseteq \Phi ,
\end{equation*}%
with the same multiplicities. This can be viewed as the restricted root
system of the same Cartan class as $G/K$, but with rank $n-1$. For instance,
if $G/K$ is of Cartan class $AIII,$ so that 
\begin{equation*}
G/K=SU(p,n)/SU(p)\times SU(n)
\end{equation*}%
for some $p>n$, then $\Phi ^{\prime }$ is the restricted root system of the
symmetric space 
\begin{equation*}
SU(p-1,n-1)/SU(p-1)\times SU(n-1),
\end{equation*}%
of Cartan class $AIII$, Lie type $BC_{n-1}$. For the purposes of this proof,
we will call this the `reduced symmetric space'. We remark that the reduced
symmetric space has rank $n-1$ and that the multiplicities of the roots are
unchanged.

By identifying $a\in A_{0}$ with $X_{a}\in \mathfrak{a}$, we can assume $%
a=\sum_{i=1}^{n}a_{i}e_{i}$. We let $a^{\prime }=\sum_{i=2}^{n}a_{i}e_{i}$
(understood, appropriately, as an element in the reduced symmetric space)
and observe that the annihilating root system of $a^{\prime }$ is of type $%
BC_{n-2}$.

Put $\lambda ^{\prime }=\sum_{i=2}^{n}\lambda _{i}e_{i},$ so that for $%
\alpha \in \Phi ^{\prime },$ $\langle \alpha ,\lambda ^{\prime }\rangle
=\langle \alpha ,\lambda \rangle $. An elementary, but useful, observation
is that $\Phi ^{+}(a)^{c}$ consists of the union of the non-annihilating
positive roots of $a$ that belong to $\Phi ^{\prime }$ together with those
non-annihilating positive roots that do not belong to $\Phi ^{\prime }$,
namely $e_{1}\pm e_{n_{0}}$. Moreover, the non-annihilating roots which are
in $\Phi ^{\prime }$ are precisely the non-annihilating roots of $a^{\prime
} $. Thus 
\begin{equation*}
P_{n-1}(\lambda ^{\prime },k,a^{\prime })=\prod_{\alpha \in (\Psi
^{+})^{c}\cap \Phi ^{\prime +}}\left\vert 1+\langle \lambda ,\alpha \rangle
\right\vert ^{-m_{\alpha }k}\prod_{\alpha \in \Phi ^{\prime +}}\left\vert
1+\langle \lambda ,\alpha \rangle \right\vert ^{m_{\alpha }}.
\end{equation*}%
Since $\langle \lambda ,e_{1}+e_{n_{0}}\rangle \geq c\lambda _{1}$ and the
induction assumption ensures that $P_{n-1}(\lambda ^{\prime },k,a^{\prime })$
is bounded independently of $\lambda ^{\prime }$ and $k,$ we see that 
\begin{eqnarray}
P_{n}(\lambda ,k,a) &=&P_{n-1}(\lambda ^{\prime },k,a^{\prime
})\prod_{\alpha \in (\Psi ^{+})^{c}\diagdown \Phi ^{\prime +}}\left\vert
1+\langle \lambda ,\alpha \rangle \right\vert ^{-m_{\alpha }k}\prod_{\alpha
\in \Phi ^{+}\diagdown \Phi ^{\prime +}}\left\vert 1+\langle \lambda ,\alpha
\rangle \right\vert ^{m_{\alpha }}  \notag \\
&\leq &P_{n-1}(\lambda ^{\prime },k,a^{\prime })\prod_{\alpha =e_{1}\pm
e_{n_{0}}}\left\vert 1+\langle \lambda ,\alpha \rangle \right\vert
^{-m_{\alpha }k}\prod_{\substack{ \alpha =e_{1}\pm e_{j},j=2,...,n  \\ %
e_{1},2e_{1}}}\left\vert 1+\langle \lambda ,\alpha \rangle \right\vert
^{m_{\alpha }}  \notag \\
&\leq &P_{n-1}(\lambda ^{\prime },k,a^{\prime })\lambda _{1}^{\eta
_{0}(2(n-1)-k)+\eta _{1}+\eta _{2}}\leq C\lambda _{1}^{\eta
_{0}(2(n-1)-k)+\eta _{1}+\eta _{2}}.  \label{Case1b}
\end{eqnarray}

\medskip

Case (b) $\Psi $ is of type $A_{n-1}$: Hence $\Psi
^{+}=\{s_{i}e_{i}-s_{j}e_{j}:1\leq i<j\leq n\}$ for some choice of $%
s_{i}=\pm 1$. We define $\Phi ^{\prime },a^{\prime },\lambda ^{\prime }$ as
above, so that $\Phi ^{\prime }$ is type $BC_{n-1}$ and the subset of
annihilating roots of $a^{\prime }$ is of type $A_{n-2}$. Again we factor $%
P_{n}(\lambda ,k,a)$ and use the fact that $P_{n-1}(\lambda ^{\prime
},k,a^{\prime })$ is bounded to see that 
\begin{eqnarray*}
P_{n}(\lambda ,k,a) &=&P_{n-1}(\lambda ^{\prime },k,a^{\prime
})\prod_{\alpha \in (\Psi ^{+})^{c}\diagdown \Phi ^{\prime +}}\left\vert
1+\langle \lambda ,\alpha \rangle \right\vert ^{-m_{\alpha }k}\prod_{\alpha
\in \Phi ^{+}\diagdown \Phi ^{\prime +}}\left\vert 1+\langle \lambda ,\alpha
\rangle \right\vert ^{m_{\alpha }} \\
&\leq &P_{n-1}(\lambda ^{\prime },k,a^{\prime })\prod_{\substack{ \alpha
=s_{1}e_{1}+s_{j}e_{j},  \\ (2)e_{1}}}\left\vert 1+\langle \lambda ,\alpha
\rangle \right\vert ^{-m_{\alpha }k}\prod_{\substack{ \alpha =\varepsilon
_{1}\pm e_{j},  \\ (2)e_{1}}}\left\vert 1+\langle \lambda ,\alpha \rangle
\right\vert ^{m_{\alpha }} \\
&\leq &C\prod_{\alpha =s_{1}e_{1}+s_{j}e_{j},(2)e_{1}}\left\vert 1+\langle
\lambda ,\alpha \rangle \right\vert ^{-m_{\alpha }k}\cdot \lambda
_{1}^{2(n-1)\eta _{0}+\eta _{1}+\eta _{2}}.
\end{eqnarray*}%
There is no loss in assuming $s_{1}=1,$ thus%
\begin{equation*}
\prod_{\alpha =s_{1}e_{1}+s_{j}e_{j},(2)e_{1}}\left\vert 1+\langle \lambda
,\alpha \rangle \right\vert ^{-m_{\alpha }k}\leq \lambda
_{1}^{-\#\{j>1:s_{j}=1\}\eta _{0}k}\lambda _{1}^{-(\eta _{1}+\eta _{2})k}.
\end{equation*}%
If there is at least one $j>1$ such that $s_{j}=1,$ then we have 
\begin{equation}
P_{n}(\lambda ,k,a)\leq C\lambda _{1}^{\eta _{0}(2(n-1)-k)+(\eta _{1}+\eta
_{2})(1-k)},  \label{Case2a}
\end{equation}%
agreeing with (\ref{Claim1}).

So assume $s_{j}=-1$ for all $j>1$. We note that if $\alpha =e_{1}-e_{j},$
then $\left\vert 1+\langle \lambda ,\alpha \rangle \right\vert =1+\lambda
_{1}-\lambda _{j},$ so if there is some $j$ with $\lambda _{j}\leq \lambda
_{1}/2,$ then $\left\vert 1+\langle \lambda ,\alpha \rangle \right\vert \geq
\lambda _{1}/2$. Hence 
\begin{equation*}
\prod_{\alpha \in s_{1}e_{1}+s_{j}e_{j},(2)e_{1}}\left\vert 1+\langle
\lambda ,\alpha \rangle \right\vert ^{-m_{\alpha }k}\leq C\lambda
_{1}^{-\eta _{0}k-(\eta _{1}+\eta _{2})k}
\end{equation*}%
and we can obtain the same bound on $P_{n}(k,\lambda ,a)$ as in (\ref{Case2a}%
) (with a different choice of constant).

Thus we can also assume $\lambda _{j}\geq \lambda _{1}/2$ for all $j>1$.
Then we give a direct argument, rather than appealing to induction. The
choice of $s_{1}=1$ and $s_{j}=-1$ for all $j\neq 1$ means that%
\begin{equation*}
\Phi ^{+}(a)^{c}=\{e_{1}-e_{k},e_{i}+e_{j},(2)e_{t}:2\leq i<j\leq n,\text{ }%
k\geq 2,t\geq 1\}.
\end{equation*}%
Furthermore, $\left\vert 1+\langle \lambda ,e_{i}+e_{j}\rangle \right\vert
\geq \lambda _{i}+\lambda _{j}\geq \lambda _{1}$ for all $i,j\geq 2$ and
similarly $\left\vert 1+\langle \lambda ,(2)e_{t}\rangle \right\vert \geq
\lambda _{1}/2$ for all $t\geq 1$. Thus%
\begin{equation*}
\prod_{\alpha \in (\Psi ^{+})^{c}}\left\vert 1+\langle \lambda ,\alpha
\rangle \right\vert ^{m_{\alpha }}\geq \lambda _{1}^{\binom{n-1}{2}\eta
_{0}+n(\eta _{1}+\eta _{2})}.
\end{equation*}%
Coupled with (\ref{All}), this gives%
\begin{equation}
P_{n}(\lambda ,k,a)\leq C\lambda _{1}^{\eta _{0}\left( 2\binom{n}{2}-k\binom{%
n-1}{2}\right) +n(\eta _{1}+\eta _{2})(1-k)}.  \label{Case2b}
\end{equation}

It is routine to check that this implies that the claim of the lemma holds.

\medskip

Case (c) $\Psi $ is of type $BC_{n-j}\times A_{j-1}$ with $2\leq j\leq n-1$:
In this situation, there are disjoint sets of indices, $I,J\subseteq
\{1,...,n\}$ where $\left\vert I\right\vert =n-j$, $\left\vert J\right\vert
=j\geq 2,$ and a choice $s_{t}=\pm 1$ for $t\in J$ such that 
\begin{equation*}
\Psi ^{+}=\{e_{i}\pm e_{j},(2)e_{t}:i<j,t\in I\}\cup
\{s_{i}e_{i}-s_{j}e_{j}:i<j\in J\}.
\end{equation*}

We set up the usual induction/factoring argument. If $1\in I,$ then set of
annihilating roots of $a^{\prime }$ is type $BC_{n-j-1}\times A_{j-1}$ in
the reduced symmetric space of type $BC_{n-1}$ (or type $A_{n-2}$ in $%
BC_{n-1}$ if $j=n-1$). Under this assumption, $(\Psi ^{+})^{c}\diagdown \Phi
^{\prime }$ contains all the roots $a=e_{1}+e_{j}$ for $j\in J,$ and for
such $\alpha $ we have $\langle \alpha ,\lambda \rangle \geq \lambda _{1}$.
As $\left\vert J\right\vert \geq 2,$ 
\begin{equation}
P_{n}(\lambda ,k,a)\leq C\lambda _{1}^{-k\eta _{0}\left\vert J\right\vert
+2(n-1)\eta _{0}+\eta _{1}+\eta _{2}}\leq C\lambda _{1}^{\eta
_{0}(2(n-1)-2k)+\eta _{1}+\eta _{2}},  \label{Case3a}
\end{equation}%
a better bound than (\ref{Claim1}).

Otherwise, $1\in J,$ so the set of annihilating roots of $a^{\prime }$ is
type $BC_{n-j}\times A_{j-2}$ in type $BC_{n-1}$ (or $BC_{n-2}$ in $BC_{n-1}$
if $\left\vert J\right\vert =2$). Then $(\Psi ^{+})^{c}\diagdown \Phi
^{\prime }$ contains the roots $a=e_{1}+e_{i}$ for $i\in I$ and $(2)e_{1},$
hence the usual arguments gives 
\begin{eqnarray}
P_{n}(\lambda ,k,a) &\leq &C\lambda _{1}^{-k\eta _{0}\left\vert I\right\vert
-k(\eta _{1}+\eta _{2})+2(n-1)\eta _{0}+\eta _{1}+\eta _{2}}  \label{Case3b}
\\
&\leq &C\lambda _{1}^{\eta _{0}(2(n-1)-k)+(\eta _{1}+\eta _{2})(1-k)}. 
\notag
\end{eqnarray}

This completes the induction step.

\medskip

We have seen that to start the induction argument we need only prove that $%
P_{2}(\lambda ,k,a)$ is uniformly bounded for $k\geq \kappa _{3}$. Since $%
k\geq \kappa _{3}$ ensures that $2m(1-k)+M\leq 0,$ we need only prove (\ref%
{Claim2}) to see this. For $BC_{2},$ we have $\Phi ^{+}=\{e_{1}\pm e_{2},$ $%
(2)e_{1},(2)e_{2}\},$ thus 
\begin{equation*}
\prod_{\alpha \in \Phi ^{+}}\left\vert 1+\langle \lambda ,\alpha \rangle
\right\vert ^{m_{\alpha }}\leq C\lambda _{1}^{\eta _{0}}(1+\lambda
_{1}-\lambda _{2})^{\eta _{0}}\lambda _{1}^{\eta _{1}+\eta _{2}}\lambda
_{2}^{\eta _{1}+\eta _{2}}.
\end{equation*}%
The maximal root subsystems of type $BC_{2}$ are of type $BC_{1}$ with
positive roots either $(2)e_{1}$ or $(2)e_{2}$, or of type $A_{1}$ with the
positive root being either $e_{1}-e_{2}$ or $e_{1}+e_{2}$. We can analyze
each of these cases separately, using the fact that $\lambda _{1}-\lambda
_{2}\sim \lambda _{1}$ if $\lambda _{2}\leq \lambda _{1}/2,$ and $\lambda
_{2}\sim \lambda _{1}$ if $\lambda _{2}\geq \lambda _{1}/2$. The details are
left for the reader.
\end{proof}

\begin{proof}
\lbrack of Theorem \ref{TMain}] First, suppose $G/K$ is Lie type $A_{n-1}$
or $D_{n}$. In the notation of Lemma \ref{LAD}, we have 
\begin{eqnarray*}
\left\Vert \nu _{a}^{k}\right\Vert _{2}^{2} &\leq &C\int_{\overline{%
\mathfrak{a}^{\ast +}}}\left( \left\vert \phi _{\lambda }(a)\right\vert
^{k}|c(\lambda )|^{-1}\right) ^{2}d\lambda \leq C\int \min (1,\left\Vert
\lambda \right\Vert ^{-\eta _{0}p_{k}})d\lambda \\
&\leq &C\int_{1}^{\infty }t^{-\eta _{0}p_{k}}t^{n-1}dt.
\end{eqnarray*}%
Of course, the last integral is finite if $k$ is chosen so that $\eta
_{0}p_{k}>n$. Using the values obtained for $p_{k}$ in the Lemma gives the
specified choice of $k_{G}$.

A\ similar argument can be applied for types $B_{n},C_{n}$ or $BC_{n},$
using the claims of Lemma \ref{LBC}.

To prove the statement about the convolution of $k$ (possibly distinct)
orbital measures $\nu _{a_{i}},$ with $a_{i}\in A_{0},$ we use the fact that 
$\nu _{a_{i}}^{k}\in L^{2}$ for the specified choices of $k$ and apply the
generalized Holder's inequality in the same manner as we did in the proof of
Corollary \ref{Creg}.
\end{proof}

\begin{remark}
The technique of Lemma \ref{LAD} could be applied to the symmetric spaces of
type $B_{n}$ or $C_{n}$ which have the additional property that all
restricted roots have the same multiplicity. But the results are no better
than can be obtained by Lemma \ref{LBC}. The induction technique of Lemma %
\ref{LBC} could also be applied to types $A_{n}$ and $D_{n}$. This takes
much more work than Lemma \ref{LAD} and gives only modest improvements.
\end{remark}

\medskip

Similar techniques can also be applied to the symmetric spaces with root
systems of exceptional types.

\begin{proposition}
\label{excep}Suppose $G/K$ is a symmetric space with restricted root system
of exceptional Lie type $G_{2},F_{4}$, $E_{6},E_{7}$ or $E_{8}$. Then $\nu
_{a_{1}}\ast \cdot \cdot \cdot \ast \nu _{a_{k}}\in L^{2}$ if $k\geq k_{G}$
as stated in the chart. 
\begin{equation*}
\begin{array}{ccccc}
\text{Lie Type} & k_{G} &  & F_{4}\text{ - Cartan Type} & k_{G} \\ 
E_{7},E_{8} & 8 &  & EII & 7 \\ 
G_{2} & 4 &  & EVI & 11 \\ 
E_{6},F_{4}\text{ all same mult} & 7 &  & EIX & 19%
\end{array}%
\end{equation*}
\end{proposition}

\begin{remark}
For comparison, the dimension of $G/K$ is 40 for $EII,$ 64 for \ $EVI$ and
112 for $EIX$.
\end{remark}

\begin{proof}
When all the restricted roots of the symmetric space have the same
multiplicity, we reason as in the proof of Lemma \ref{LAD}, using the fact
(with the notation of that lemma) shown in \cite{HY} that $s=1/(n-1)$ if the
Lie type is $E_{n},$ $s=1/5$ for type $F_{4}$ and $s=2/5$ for type $G_{2}$.

For the final three cases (Lie type $F_{4},$ Cartan types $EII$, $EVI$ or $%
EIX$) we note that the maximal annihilating root systems are types $%
A_{1}\times A_{2}$, $A_{1}\times B_{2}$, $A_{1}\times C_{2}$, $A_{1}\times
A_{1}\times A_{1}$, $B_{3}$ and $C_{3},$ all of which have cardinality at
most $9$, and do a counting argument similar to that done in the proof of
Theorem \ref{regular}.
\end{proof}

\begin{remark}
It would be interesting to know the sharp $L^{2}$ results and whether the $%
L^{1}-$ $L^{2}$ dichotomy only fails for the symmetric space of Lie type $%
A_{1}$ and Cartan class $AI$.
\end{remark}

\subsection{Differentiability properties}

If $\nu ^{k}\in L^{2},$ then $\nu ^{2k}=\nu ^{k}\ast \nu ^{k}$ has a
continuous density function. However, more can be said about the smoothness
of these measures using following fact proven in \cite[Prop. 3.1(vi)]{Ank}: 
\begin{equation*}
\left\vert \frac{d^{m}}{dt^{m}}\left( \phi _{\lambda }(g\exp (tX)\right)
|_{t=0}\right\vert \leq C_{1}(1+\left\Vert \lambda \right\Vert )^{m}.
\end{equation*}%
In proving Theorems \ref{regular} and \ref{TMain}, we have seen that there
are constants $C$ and $q(k)$ such that $\left\vert \left( \phi _{\lambda
}(a)\right) ^{k}c(\lambda )^{-1}\right\vert ^{2}\leq C\min (1,\left\Vert
\lambda \right\Vert ^{q(k)})$ for all $\lambda $. Thus, with $n=rankG/K,$ we
have 
\begin{eqnarray*}
\int_{\mathfrak{a}^{\ast }}\left\vert \phi _{\lambda }(a)^{k}\frac{d}{dt}%
\left( \phi _{\lambda }(g\exp (tX)\right) |_{t=0}\left\vert c(\lambda
)\right\vert ^{-2}\right\vert d\lambda &\leq &C\int_{\overline{\mathfrak{a}%
^{\ast +}}}\left\vert \phi _{\lambda }(a)^{k/2}c(\lambda )^{-1}\right\vert
^{2}(1+\left\Vert \lambda \right\Vert ) \\
&\leq &C\int_{\overline{\mathfrak{a}^{\ast +}}}\left\Vert \lambda
\right\Vert ^{q(k/2)}(1+\left\Vert \lambda \right\Vert ) \\
&\leq &C\int_{1}^{\infty }t^{n-1+q(k/2)+1}
\end{eqnarray*}%
and this is finite provided $n+q(k/2)<-1$. If $k$ is chosen sufficiently
large that this inequality holds, then Leibniz's rule applied to the
Inversion formula (\cite[IV Thm. 7.5]{He}) shows that%
\begin{equation*}
X\nu _{a}^{k}(g)=\frac{1}{\left\vert W\right\vert }\int_{\mathfrak{a}^{\ast
}}\phi _{\lambda }(a)^{k}\frac{d}{dt}\left( \phi _{\lambda }(g\exp
(tX)\right) |_{t=0}\left\vert c(\lambda )\right\vert ^{-2}d\lambda
\end{equation*}%
is well defined and hence $\nu _{a}^{k}$ is differentiable. More generally, $%
\nu _{a}^{k}$ is $r$-times differentiable if $n-1+q(k/2)+r<-1$.

For example, if $G/K$ is Lie type $A_{n},$ then Lemma \ref{LAD} yields $%
q(k/2)\leq \eta _{0}(n-k/2)$. Thus we have $\nu _{a}^{k}$ is differentiable
for all $a\in A_{0}$ if $k>2n+2(n+1)/\eta _{0}$. If $a$ is a regular element
and $G/K$ is not of Lie type $A_{1}$ and Cartan type $AI$, then one can
similarly use Lemma \ref{L:reg} to check that $\nu _{a}^{k}$ is
differentiable if $k>4$. Similar statements can be made about higher orders
of differentiability. These observations improve upon \cite{AReg} where it
was shown that if $a$ was a regular element, then $\nu _{a}^{k}$ is
differentiable for $k>\dim G/K+1.$

\section{Appendix}

In the charts below we summarize some of the important facts about these
symmetric spaces. These are taken from \cite{CM} and \cite[Ch. X]{Hediff}.

\medskip

\frame{$%
\begin{array}{ccccc}
\begin{array}{c}
\text{Restricted} \\ 
\text{root space}%
\end{array}
& 
\begin{array}{c}
\text{Cartan} \\ 
\text{class}%
\end{array}
& G/K & \dim G/K & 
\begin{array}{c}
\text{Multiplicities} \\ 
\eta _{0};\eta _{1};\eta _{2}%
\end{array}
\\ 
A_{n-1} & AI & \frac{SL(n,R)}{SO(n)} & \frac{1}{2}(n-1)(n+2) & 1;0;0 \\ 
A_{n-1} & AII & \frac{SL(n,H)}{Sp(n)} & (n-1)(2n+1) & 4;0;0 \\ 
\begin{array}{c}
BC_{n},p>n \\ 
C_{n},p=n%
\end{array}
& AIII & \frac{SU(p,n)}{SU(p)\times SU(n)} & 2pn & 2;1;2(p-n) \\ 
C_{n} & CI & \frac{Sp(n,R)}{SU(n)} & n(n+1) & 1;1;0 \\ 
\begin{array}{c}
BC_{n},p>n \\ 
C_{n},p=n%
\end{array}
& CII & \frac{Sp(p,n)}{Sp(p)\times Sp(n)} & 4pn & 4;3;4(p-n) \\ 
C_{n} & DIII\text{ (even)} & \frac{SO^{\ast }(4n)}{U(2n)} & 2n(2n-1) & 4;1;0
\\ 
BC_{n} & DIII\text{ (odd)} & \frac{SO^{\ast }(4n+2)}{U(2n+1)} & 2n(2n+1) & 
4;1;4 \\ 
\begin{array}{c}
B_{n},p>n \\ 
D_{n},p=n%
\end{array}
& BDI & \frac{SO_{0}(p,n)}{SO(p)\times SO(n)} & pn & 1;0;p-n%
\end{array}%
$}

\medskip

\frame{$%
\begin{array}{cccc}
\begin{array}{c}
\text{Restricted} \\ 
\text{root space}%
\end{array}
& \text{Cartan class} & \dim G/K & \text{Multiplicities} \\ 
BC_{2} & EIII & 32 & 8;6;1 \\ 
A_{2} & EIV & 26 & 8 \\ 
C_{3} & EVII & 54 & 8;1 \\ 
BC_{1} & FII & 16 & 8;7%
\end{array}%
$}

\medskip

\frame{$%
\begin{array}{cccc}
\begin{array}{c}
\text{Restricted} \\ 
\text{root space}%
\end{array}
& \text{Cartan class} & \dim G/K & \text{Multiplicities} \\ 
G_{2} & G & 8 & 1 \\ 
F_{4} & 
\begin{array}{c}
EII \\ 
EVI \\ 
EIX \\ 
FI%
\end{array}
& 
\begin{array}{c}
40 \\ 
64 \\ 
112 \\ 
28%
\end{array}
& 
\begin{array}{c}
1,2 \\ 
1,4 \\ 
8,1 \\ 
1%
\end{array}
\\ 
E_{6} & EI & 42 & 1 \\ 
E_{7} & EV & 70 & 1 \\ 
E_{8} & EVIII & 128 & 1%
\end{array}%
$}

\end{document}